\documentclass[11pt]{amsart}

\usepackage[a4paper,hmargin=3.5cm,vmargin=3.5cm]{geometry}
\usepackage{amsfonts,amssymb,amscd,amstext}
\usepackage{graphicx}
\usepackage[dvips]{epsfig}

\usepackage{fancyhdr}
\usepackage{times}

\usepackage{enumerate}
\usepackage{titlesec}
\usepackage{mathrsfs}
\usepackage{stmaryrd}

\def\Rcal{\mathcal{R}}

\def\Ocal{\mathcal{O}}

\def\bc{\boldsymbol{\mathfrak{B}}}

\def\c{\mathbb{C}}
\def\r{\mathbb{R}}
\def\n{\mathbb{N}}
\def\z{\mathbb{Z}}

\def\b{\mathbb{B}}
\def\d{\mathbb{D}}
\def\s{\mathbb{S}}

\def\mgot{\mathfrak{m}}

\def\pgot{\mathfrak{p}}

\def\cgot{\mathfrak{c}}
\def\Fgot{\mathfrak{F}}

\def\agot{\mathfrak{a}}

\def\oB{\overline{\mathscr{T}}}

\def\dist{\mathrm{dist}}

\titleformat{\section}
{\filcenter\bfseries\large} {\thesection{.}}{0.2cm}{}
\titleformat{\subsection}[runin]
{\bfseries} {\thesubsection{.}}{0.15cm}{}[.]
\titleformat{\subsubsection}[runin]
{\em}{\thesubsubsection{.}}{0.15cm}{}[.]

\usepackage[up,bf]{caption}

\newtheorem{theorem}{Theorem}[section]

\newtheorem{claim}[theorem]{Claim}
\newtheorem{lemma}[theorem]{Lemma}
\newtheorem{corollary}[theorem]{Corollary}

\newtheorem{definition}[theorem]{Definition}
\newtheorem{question}[theorem]{Question}

\newtheorem*{question*}{Question}

\theoremstyle{definition}

\numberwithin{equation}{section}
\numberwithin{figure}{section}

\usepackage{color}

\begin{document}
\fancyhead[LO]{Complete complex hypersurfaces in the ball}
\fancyhead[RE]{A.\ Alarc\'on, J.\ Globevnik, and F.\ J.\ L\'opez}
\fancyhead[RO,LE]{\thepage}

\thispagestyle{empty}

\vspace*{7mm}
\begin{center}
{\bf \LARGE A construction of complete complex hypersurfaces in the ball with control on the topology}
\vspace*{5mm}

{\large\bf A.\ Alarc\'on, J.\ Globevnik, and F.\ J.\ L\'opez}
\end{center}


\vspace*{7mm}

\begin{quote}
{\small
\noindent {\bf Abstract}\hspace*{0.1cm}
Given a 
closed complex hypersurface $Z\subset \c^{N+1}$ $(N\in\n)$ and a compact subset $K\subset Z$, we prove the existence of a pseudoconvex Runge domain $D$ in $Z$ such that $K\subset D$ and there is a complete proper holomorphic embedding from $D$ into the unit ball of $\c^{N+1}$.
For $N=1$, we derive the existence of complete properly embedded complex curves in the unit ball of $\c^2$, with arbitrarily prescribed finite topology. In particular, there exist complete proper holomorphic embeddings of the unit disc $\d\subset \c$ into the unit ball of $\c^2$. 

These are the first known examples of complete bounded embedded complex hypersurfaces in $\c^{N+1}$ with any control on the topology. 



}
\end{quote}


\section{Introduction and main results}\label{sec:intro}

Let $\d$ denote the unit disc in $\c$ and, for $N\in\n$, denote by $\bc_{N+1}$ the unit ball in $\c^{N+1}$.

In 1977 P.\ Yang asked whether there exist complete immersed complex submanifolds $\varphi\colon M^k\to\c^{N+1}$ $(k\leq N)$ with bounded image \cite{Yang1, Yang2}. 
Here, {\em complete} means that the Riemannian manifold $(M,\varphi^*ds)$ is complete, where $ds$ is the Euclidean metric in $\c^{N+1}$; equivalently, the image by $\varphi$ of every divergent path in $M$ has infinite Euclidean length.
 
P.\ Jones \cite{Jones}  was the first to construct a bounded complete holomorphic {\em immersion} $\d\rightarrow \c^2$, a bounded complete holomorphic {\em embedding} $\d\hookrightarrow \c^3$, and a proper complete holomorphic embedding $\d\hookrightarrow \bc_4$. 
Jones' pioneering results have been extended to the existence of proper complete  holomorphic immersions $\Rcal\rightarrow \bc_2$ and embeddings $\Rcal\hookrightarrow \bc_3$, where $\Rcal$ is either an open Riemann surface of arbitrary topology (see Alarc\'{o}n and L\'{o}pez \cite{AL-CY} and Alarc\'on and Forstneri\v c \cite{AF-2}), or a given bordered Riemann surface (see Alarc\'{o}n and  Forstneri\v c \cite{AF-1,AF-3}). Further, here $\bc_2$ and $\bc_3$ may be replaced by any convex domain in $\c^2$ and $\c^3$, respectively. Moreover, given $k\in\n$, an easy application of these results furnishes bounded complete holomorphic immersions $\Rcal^k:=\Rcal\times \stackrel{k}{\cdots}\times \Rcal\to\c^{2k}$ and embeddings $\Rcal^k\hookrightarrow \c^{3k}$ \cite{AF-1}. In the same direction, B.\ Drinovec Drnov\v sek \cite{D} recently proved that every bounded strictly pseudoconvex  domain $D\subset\c^k$ with $\mathscr{C}^2$-boundary admits a complete proper holomorphic embedding $D\hookrightarrow \bc_{N+1}$ provided that the codimension $N+1-k$ is large enough.

To find complete bounded holomorphic {\em embeddings} $\varphi\colon M^N\hookrightarrow \c^{N+1}$ is considerably more difficult. For instance, the self-intersection points of hypersurfaces in $\c^{N+1}$ are generic and thus cannot be removed by small perturbations, which 
is possible when the codimension is large enough (to be precise, when $2k\leq N$). So, any induction method for constructing such hypersurfaces will have to take this into account and at no step create self-intersection points.  

The embedding problem was settled in the lowest dimensional case by Alarc\'{o}n and L\'{o}pez, who proved that every convex domain in $\c^2$ contains a complete properly embedded complex curve \cite{AL-JEMS}. Their examples come after a recursive construction process which applies, at each step, a self-intersection removal procedure consisting of replacing every normal crossing 
in a complex curve by an embedded annulus. This process, which must be done while ensuring the completeness of the limit curve, is very delicate and does not provide any control on the topology of the curve. In principle, the examples could be of very complicated topology.

After this, a different approach was used by Globevnik \cite{Glob1}.  He found an embedded complete holomorphic curve as a level set of some wildly oscillating holomorphic function on $\bc_2$.  This construction worked also in higher dimensions which lead to the complete solution of the Yang problem in all dimensions by proving that, for any $N\in \n$, there is a complete, closed complex hypersurface in $\bc_{N+1}\subset \c^{N+1}$. The same holds when replacing $\bc_{N+1}$ by any pseudoconvex  domain in $\c^{N+1}$ \cite{Glob2}. Again, this procedure does not supply any information about the topology of the hypersurface, which could be very involved.

 So at the moment there are two different methods to prove the existence of complete, closed hypersurfaces in  $\bc_{N+1}$ when $N=1$, and one when $N\geq 2$. Neither of these methods provides any information about the topology of such a hypersurface.  In this paper we develop a conceptually new technique for constructing complete closed complex hypersurfaces in the unit ball $\bc_{N+1}\subset\c^{N+1}$, $N\in\n$, which in addition admits control on the topology of the examples. In particular, we show that there is a complete proper holomorphic embedding $\d\hookrightarrow\bc_2$  (see Corollary \ref{co:main2} below) and thus answer a question left open in \cite{AL-JEMS,Glob1}.

Before stating our results we need some background.
Given a Stein manifold $X$, we denote by $\Ocal(X)$ the algebra of holomorphic functions $X\to\c$. A domain (open and connected subset) $\Omega\subset X$  is called a {\em pseudoconvex domain} if it has a strongly plurisubharmonic exhaustion function. In particular, $\Omega$ endowed  with the induced complex structure is a Stein manifold as well. A domain $\Omega\subset X$ is said to be a {\em Runge domain in $X$} if every holomorphic function $\Omega\to\c$ can be uniformly approximated on compact subsets of $\Omega$ by functions in $\Ocal(X)$.

%
%

Our main result may be stated as follows.
\begin{theorem}\label{th:main}
Let $Z\subset \c^{N+1}$ $(N\in\n)$ be a  
closed complex hypersurface such that $Z\cap \bc_{N+1}\neq\emptyset$, let $K\subset Z\cap \bc_{N+1}$ be a connected compact subset, and let $\epsilon>0$. There exists  a pseudoconvex 
domain $D\subset Z$ with the following properties:
\begin{enumerate}[{\rm (i)}]
\item  $K\subset D$.
\item $D$ is Runge in $Z$.
\item There exists a complete proper holomorphic embedding $\psi\colon D\hookrightarrow\bc_{N+1}$ such that $|\psi(\zeta)-\zeta|<\epsilon$ for all $\zeta\in K$. 
\end{enumerate}

In particular, $\bc_{N+1}$ contains complete closed complex hypersurfaces which are biholomorphic to a pseudoconvex Runge domain in $\c^N$.
\end{theorem}
These are the first known examples of complete bounded complex hypersurfaces in $\c^{N+1}$, $N\in\n$, for which one has any topological information. 
In $\c^2$  we have the following more precise result:
\begin{corollary}\label{co:main2}
Let $Z\subset \c^2$ be a properly embedded complex curve and let $K\subset Z\cap \bc_2$ be a compact connected subset. 
Given $\epsilon>0$ there exist a Runge domain $D\subset Z$  and a complete proper holomorphic embedding $\psi\colon D\hookrightarrow \bc_2$ such that $K\subset D$ and $|\psi(x)-x|<\epsilon$ for all $x\in K$.

As a consequence,  the unit ball $\bc_2$ of $\c^2$ carries complete properly embedded complex curves with any finite topology. In particular, there are proper complete holomorphic embeddings $\d\hookrightarrow \bc_2$.
\end{corollary}
The second part of the corollary follows from the well-known fact that  all surfaces of finite topology can be realized as Runge domains of properly embedded complex curves in $\c^2$ (see \v Cerne and Forstneri\v c \cite{CF} and Section \ref{sec:main}).

Our technique is different from the ones in \cite{AL-JEMS,Glob1}. We begin with a 
closed complex hypersurface $Z\subset \c^{N+1}$, intersecting $\bc_{N+1}$, a compact subset $K\subset Z\cap \bc_{N+1}$,  and the natural embedding $Z \hookrightarrow \c^{N+1}$ given by the inclusion map. In a recursive process, we compose this initial embedding with a sequence of
holomorphic automorphisms of $\c^{N+1}$ which converges uniformly in compact subsets of $\bc_{N+1}$. In this way we obtain a sequence of proper holomorphic embeddings  $Z \hookrightarrow \c^{N+1}$ whose images converge uniformly on compact subsets of $\bc_{N+1}$ to a closed embedded complex hypersurface of $\bc_{N+1}$. Moreover, carrying out this process in the right way, we may ensure that a connected component $D$ of the resulting hypersurface is biholomorphic to a pseudoconvex Runge domain in $Z$ containing $K$, and is closed in $\bc_{N+1}$ and complete.

This program will be performed in two different steps, which we now describe.
For a while we shall be working in real Euclidean space $\r^{n+1}$, $n\in\n$. We let
 $\b_{n+1}$ and $\s^{n}=b\b_{n+1}$ denote the unit ball and the unit sphere in $\r^{n+1}$ of center $0$ and radius $1$; with this notation, $\bc_{N+1}=\b_{2N+2}$ for all $N\in\n$.
We denote by $\langle \cdot,\cdot \rangle$ and $|\cdot|$ the Euclidean scalar product and norm in $\r^{n+1}$. We write $q+r C$ for the set $\{q+r p\colon p\in C\}\subset\r^{n+1}$ for $r>0$, $q\in\r^{n+1}$, and $C\subset \r^{n+1}$. Finally, given $p\in\r^{n+1}\setminus\{0\}$, we set $\langle p\rangle^\bot:=\{q\in\r^{n+1}\colon \langle q,p\rangle=0\}$. Observe that $p+\langle p\rangle^\bot$ is the affine tangent hyperplane to the sphere $|p|\s^n$ at the point $p$. The following objects play a fundamental role in our construction.

\begin{definition}\label{def:balls}
Given $p\in\r^{n+1}\setminus\{0\}$ and $r>0$, the set
\[
\oB(p,r):=p+(r\overline \b_{n+1}\cap \langle p\rangle^\bot)
\]
will be called {\em the (closed) tangent ball in $\r^{n+1}$ of center $p$ and radius $r$}.
\end{definition}

Thus, for $p\in\r^{n+1}\setminus\{0\}$ and $r>0$, the tangent ball $\oB(p,r)$ is the closed ball of center $p$ and radius $r$ in the affine hyperplane $p+\langle p\rangle^\bot$. Given a collection $\Fgot$ of tangent balls in $\r^{n+1}$, we denote
\[
|\Fgot|:=\bigcup_{T\in\Fgot} T \subset \r^{n+1}.
\]
We will gather tangent balls in $\r^{n+1}$ into collections which we call {\em tidy} according to the following
\begin{definition}\label{def:tidy}
A collection $\Fgot=\{\oB(p_j,r_j)\}_{j\in J}$ of tangent balls in $\b_{n+1}\subset\r^{n+1}$ will be called  {\em a tidy collection} if the following conditions are satisfied:
\begin{itemize}
\item $t\overline\b_{n+1}$ intersects finitely many balls in $\Fgot$ for all $0<t<1$.
\item If $\oB(p_1,r_1),\oB(p_2,r_2)\in\Fgot$ and $|p_1|=|p_2|$, then $r_1=r_2$ and $\oB(p_1,r_1)\cap\oB(p_2,r_2)=\emptyset$.
\item If $\oB(p_1,r_1),\oB(p_2,r_2)\in\Fgot$ and $|p_1|<|p_2|$, then $\oB(p_1,r_1)\subset |p_2|\b_{n+1}$.
\end{itemize}
In particular, $\Fgot$ consists of at most countably many pairwise disjoint tangent balls and $|\Fgot|$ is a proper subset of $\b_{n+1}$.
\end{definition}

Thus, all the balls in a tidy collection $\Fgot$ which are tangent to the same sphere have the same radius and are pairwise disjoint; if $\Fgot$ contains balls tangent to two spheres then all the balls tangent to the smaller sphere are contained in the open ball whose boundary is the larger sphere. If $\Fgot'\subset \Fgot$ consists of those balls in $\Fgot$ which are tangent to a given sphere $\lambda \s^n$, $0<\lambda<1$, then $\Fgot'$ is also tidy and  there is $\mu\in ]\lambda,1[$ such that the boundary $bT$ lies in $\mu \s^n$ for any $T\in \Fgot'$. In particular, for each $T\in \Fgot'$ the affine hyperplane of $\r^{n+1}$ containing $T$ is disjoint from the convex hull of $|\Fgot'|\setminus T$. This simple property of tidy collections will be crucial in our construction.

Tangent balls  in $\b_{n+1}$ may be viewed as obstacles on the way towards the boundary $\s^n$ when we want to reach $\s^n$ along a path in $\b_{n+1}$ that misses all the balls in a collection. The first step in the proof of Theorem \ref{th:main} is to construct a tidy collection $\Fgot$ of tangent balls in $\b_{n+1}$ in such a way that every path $\gamma\colon [0,1[ \rightarrow \b_{n+1}$ satisfying $\lim _{t\rightarrow 1}|\gamma(t)| = 1$ and missing all the balls of $\Fgot$, has infinite length. 

\begin{theorem}[Building obstacles]
\label{th:infinity}
Let $n\in\n$ and $0<\lambda_0<1$.
There exists a tidy collection $\Fgot$ of tangent balls in $\b_{n+1}\subset\r^{n+1}$ satisfying $|\Fgot|\cap \lambda_0 \overline\b_{n+1}=\emptyset$ and having the following property: Every path $\gamma\colon[0,1[\to\b_{n+1}$ such that 
$\lim _{t\rightarrow 1}|\gamma(t)| = 1$ and $\gamma([0,1[)\cap|\Fgot|=\emptyset$, has infinite length. 

In particular, every closed submanifold of $\b_{n+1}$ missing $|\Fgot|$ is complete.
\end{theorem}

Theorem \ref{th:infinity} will be proved in Sec.\ \ref{sec:infinity} (see the more general Theorem \ref{th:infinity2}).
It is clear that a collection $\Fgot$ as in Theorem \ref{th:infinity} will consists of infinitely many tangent balls. 

The second step in the proof of Theorem \ref{th:main} is to construct a proper holomorphic embedding from a domain $D$ in any given 
closed complex hypersurface $Z\subset \c^{N+1}$ to $\bc_{N+1}$, whose image misses all the balls in a given tidy collection $\Fgot$ of tangent balls in $\bc_{N+1}$. 

\begin{theorem}[Avoiding obstacles]
\label{th:disc}
Let $N\in\n$.
Let $\Fgot$ be a tidy collection of tangent balls in the open unit ball $\bc_{N+1}\subset\c^{N+1}$, and  choose $\lambda_0>0$ such that $|\Fgot|\cap \lambda_0 \overline\bc_{N+1} =\emptyset$. Let $Z\subset \c^{N+1}$ be a  closed complex hypersurface such that $Z\cap\lambda_0\overline\bc_{N+1}\neq\emptyset$ and let $\epsilon>0$.  

There exist a Runge domain $\Omega\subset \c^{N+1}$ containing $\lambda_0 \overline\bc_{N+1}$  and a biholomorphic map $\Psi\colon \Omega\to \bc_{N+1}$ such that:
\begin{enumerate}[\rm (i)]
\item $|\Psi(\zeta)-\zeta|<\epsilon$ for all $\zeta \in\lambda_0 \overline\bc_{N+1}$.
\item $\Psi(\Omega\cap Z)\cap|\Fgot|=\emptyset$.
\end{enumerate}

In particular, every connected component $D$ of $\Omega\cap Z\neq \emptyset$ is a pseudoconvex Runge domain in $Z$  and $\Psi|_D\colon D\hookrightarrow\bc_{N+1}$ is a proper holomorphic embedding whose image misses $|\Fgot|$.
\end{theorem}

The proof of Theorem \ref{th:disc} is included in Sec.\ \ref{sec:disc} and roughly goes as follows. Observe first that, given an infinite tidy collection $\Fgot=\{\oB(p_j,r_j)\colon j\in\n\}$ of tangent balls in $\bc_{N+1}$, we may assume  that the convex hull of $\bigcup_{i\leq j} \oB(p_i,r_i)$ is disjoint from $\bigcup_{i> j} \oB(p_i,r_i)$ for each $j\in \n$. Indeed, it suffices to order $\Fgot$ so that  $|p_i|\leq |p_j|$ if $i<j$.  Thus, there exists an exhaustion $\lambda_0 \overline \bc_{N+1} \Subset E_1 \Subset E_2\Subset \cdots \Subset \bigcup_{j\in \z_+} E_j=\bc_{N+1}$ of $\bc_{N+1}$ by smoothly bounded, strictly convex domains, such that $\overline E_1\cap |\Fgot|=\emptyset$, and for each $j\in \n$, $ \bigcup_{i\leq j} \oB(p_i,r_i)\subset E_{j+1}$  and $\overline E_{j+1}\cap (\bigcup_{i>j} \oB(p_i,r_i))=\emptyset$.  The key in the proof of Theorem \ref{th:disc} is to show that, given a closed complex hypersurface $X\subset\c^{N+1}$ which does not intersect $\bigcup_{i\leq j-1} \oB(p_i,r_i)$, there exists a holomorphic automorphism $\Psi_j\colon\c^{N+1}\to\c^{N+1}$ such that $\Psi_j$ is close to the identity in the compact convex set $\overline E_{j}$ and such that $\Psi_j(X)$, which is again a closed complex hypersurface, does not intersect $\bigcup_{i\leq j}\oB(p_i,r_i)$; cf.\ Lemma \ref{lem:shear1}. We obtain the biholomorphic map $\Psi\colon \Omega \to \bc_{N+1}$ in Theorem \ref{th:disc} as the limit of a sequence of automorphisms $\{\Psi_j\}_{j\in \n}$ generated in a  recursive way by application of this result.

Theorems \ref{th:infinity} and \ref{th:disc} easily imply Theorem \ref{th:main}; see Sec.\ \ref{sec:main}. 
Indeed, Theorem \ref{th:disc} applied to a tidy collection of tangent balls in $\bc_{N+1}$ given by Theorem \ref{th:infinity}, provides a complete, closed complex hypersurface in $\bc_{N+1}$, which is biholomorphic to a pseudoconvex Runge domain in a given closed complex hypersurface $Z\subset\c^{N+1}$.


\section{The Building Obstacles Theorem}\label{sec:infinity} 

In this section we prove Theorem \ref{th:infinity} in a more general form; see Theorem \ref{th:infinity2} below.

Throughout the section we fix $n\in\n$, write $\b$ for the open unit ball $\b_{n+1}\subset \r^{n+1}$ and denote by $\pgot\colon \r^{n+1}\setminus\{0\}\to \s^n=b\b$ the radial projection 
\[
\pgot(x)=\frac{x}{|x|},\quad x\in \r^{n+1}\setminus\{0\}.
\]
We denote by $\dist(\cdot,\cdot)$ and ${\rm diam}(\cdot)$ the Euclidean distance and diameter in $\r^{n+1}$, and $\imath:=\sqrt{-1}$.

We begin with the following result, which is the key in the proof of Theorem \ref{th:infinity}.

\begin{lemma}
\label{lem:points}
There exist numbers $\mgot_n\in \n$, $\mgot_n\geq 2$, and $\cgot_n\in\r$, $0<\cgot_n<1/2$, such that the following assertion holds.

For every real number $r>0$ there exist  finite subsets $F_1,\ldots,F_{\mgot_n}$ of       
 $\s^n$ such that:
\begin{enumerate} [\rm (i)]
\item  $|p-q|\geq r$ for all $p,q\in F_j$, $p\neq q$, $j=1,\ldots,\mgot_n$.
\item  If $F:=\bigcup_{j=1}^{\mgot_{n}} F_j$ then $F\neq\emptyset$ and  $\dist(p,F)\leq \cgot_{n} r$ for all $p\in \s^n$. 
\end{enumerate}
\end{lemma}
We emphasize that the numbers $\mgot_n$ and $\cgot_n$ in the lemma only depend on $n$, the dimension of the sphere. Possibly some of the $F_j$'s are empty. If $r> 2={\rm diam}(\s^n)$, condition {\rm (i)} implies that $F_j$ consists of at most one point for all $j\in\{1,\ldots,\mgot_n\}$. 

Let us outline how Theorem \ref{th:infinity} will follow from Lemma \ref{lem:points}. We will pick any sequence of numbers $0<s_0=\lambda_0<s_1<\cdots<\lim_{j\to+\infty} s_j=1$ such that $\sum_{j\in\n}\sqrt{s_j-s_{j-1}}=+\infty$; here $\lambda_0$ is given in the statement of the theorem. For each $j\in\n$, we will consider the numbers $s_{j,k}:=s_{j-1}+k\frac{s_j-s_{j-1}}{\mgot_n+1}$, $k=1,\ldots,\mgot_n+1$ (hence $s_{j-1}<s_{j,k}\leq s_j$), and will take $r_j>0$ such that the tangent ball $\oB(p,r_j)$ is contained in $s_{j,k+1}\b$ for all $p\in s_{j,k}\s^n$, $k=1,\ldots,\mgot_n$. Basic trigonometry gives that $r_j$ can be taken to be larger than $a\sqrt{s_j-s_{j-1}}$ for some constant $a>0$ which does not depend on $j$. We will then apply Lemma \ref{lem:points} for $r=2r_j$ and obtain subsets $F_{j,1},\ldots,F_{j,\mgot_n}\subset \s^n$, $j\in\n$.  Lemma \ref{lem:points} {\rm (i)} and the choice of $r_j$ will ensure that the collection of tangent balls $\Fgot=\bigcup_{j\in\n} \big(\bigcup_{k=1}^{\mgot_n}\Fgot_{j,k}\big)$, where $\Fgot_{j,k}=\{s_{j,k}\oB(p,r_j)\colon p\in F_{j,k}\}$ for all $j\in\n$ and $k\in\{1,\ldots,\mgot_n\}$, is tidy. On the other hand, Lemma \ref{lem:points} {\rm (ii)} will guarantee the existence of a constant $c>0$, which does not depend on $j$, such that for any $p\in\s^n$ the spherical ball $(p+cr_j\overline\b)\cap\s^n$ is contained in $\pgot(T)$ for some $T\in\Fgot_j:=\bigcup_{k=1}^{\mgot_n}\Fgot_{j,k}$. This implies that the length of every path $\gamma_j\colon[0,1]\to s_j\overline\b\setminus s_{j-1}\b$ such that $|\gamma_j(0)| = s_{j-1}$, $|\gamma_j(1)|=s_j$, and $\gamma_j([0,1])\cap|\Fgot_j|=\emptyset$, is at least $s_{j-1}c r_j> a s_0c\sqrt{s_j-s_{j-1}}$. (See Lemma \ref{lem:r1r2}.) Finally, since $a$, $s_0$, and $c$ do not depend on $j\in\n$, the choice of the $s_j$'s will ensure that every path $\gamma\colon[0,1[\to\b$ such that 
$\lim _{t\rightarrow 1}|\gamma(t)| = 1$ and $\gamma([0,1[)\cap|\Fgot|=\emptyset$, has infinite length. Thus, the collection $\Fgot$ of tangent balls will prove Theorem \ref{th:infinity}.

Before starting with the proof of Lemma \ref{lem:points} observe the following 
\begin{claim} \label{cla:t}
Assume that  Lemma \ref{lem:points} holds.  Then the same statement is valid with the same numbers $\mgot_{n}\in\n$ and $0<\cgot_{n}<1/2$ if we replace $\s^{n}$ by $t\s^{n}=\{tp\colon p\in\s^{n}\}$ for any $t>0$. 
\end{claim}
\begin{proof}
Pick $t, r>0$. Since we are assuming that the lemma holds for $\s^{n}$, there exist $\mgot_{n}$ subsets $C_1, \ldots, C_{\mgot_{n}}$ of $\s^{n}$ satisfying {\rm (i)} and {\rm (ii)} for the real number $r/t>0.$  Therefore, the subsets $t C_1, \ldots, t C_{\mgot_{n}}$ of $t \s^{n}$ meet {\rm (i)} and {\rm (ii)} for the real number $r$, proving the claim.
\end{proof}

\begin{proof}[Proof of Lemma \ref{lem:points}]
We proceed by induction on $n$, the dimension of the sphere. 

The basis of the induction $(n=1)$ admits several proofs. The one we include here will help the reader to understand the main point of the induction step. We let $\mgot_1=10$ and $\cgot_1=1/3$ (these numbers are not sharp but fit well with the argument in the inductive step). Choose $r>0$. It suffices to find ten subsets of $\s^1$ satisfying conditions {\rm (i)} and {\rm (ii)}. 
We distinguish cases.

Assume $r\geq 2={\rm diam}(\s^1)$. Set $F_j:=\{e^{\imath (j-1) \frac{\pi}5}\}\subset\s^1\subset\c$, $1\leq j\leq 10$. Condition {\rm (i)} trivially holds since all the $F_j$'s are unitary, whereas {\rm (ii)} follows from the equation
\begin{equation}\label{eq:distangle}
|x-x e^{\imath\sigma}|=|1-e^{\imath\sigma}|=2\sin\big(\frac{\sigma}2\big)\quad \text{for all $x\in\s^1\subset\c$, $\sigma\in[0,2\pi]$}.
\end{equation}
Indeed, given $p=e^{\imath t}\in\s^1$, $t\in[0,2\pi]$, there exists $j_0\in\{1,\ldots,10\}$ such that $|t-(j_0-1) \frac{\pi}5|<\frac{\pi}5$, and so
\[
\dist\Big(p,\bigcup_{j=1}^{10} F_j \Big)  \leq  \dist(p,F_{j_0})
\leq \big|p-pe^{\imath ((j_0-1) \frac{\pi}5-t)}\big|\leq  2\sin\big(\frac{\pi}{10}\big) < \frac23 \leq \cgot_1r;
\]
take into account that the sinus in increasing in $[0,\frac{\pi}2]$.

Assume now  $0<r<2$. In this case \eqref{eq:distangle} ensures that
\begin{equation}\label{eq:alphabasis}
|x-x e^{\imath \sigma}|\geq r,\quad x\in\s^1, \quad \text{for all $\sigma\in[\alpha,\pi]$,}
\end{equation}
where $\alpha:=2 \arcsin(r/2)\in]0,\pi[$. Set $\beta:=2 \arcsin(r/6)$ and observe that 
\begin{equation}\label{eq:betabasis}
0<\beta<\alpha<5\beta.
\end{equation}
Moreover, since the sinus is increasing in $[0,\frac{\pi}2]$, \eqref{eq:distangle} ensures that
\begin{equation}\label{eq:distangle2}
|x-x e^{\imath \sigma}|\leq \frac{r}3=\cgot_1r,\quad x\in\s^1,\quad \text{for all $\sigma\in[0,\beta]$}.
\end{equation}
For each $j\in\{1,\ldots,5\}$, consider the subset
\begin{equation}\label{eq:Fx}
F_j:=\{-\imath e^{\imath (j-1)\beta} e^{\imath 5(k-1)\beta}\colon k\in\n,\; ((j-1)+5(k-1))\beta\leq\pi\}\subset\s^1 
\end{equation}
generated by rotating the point $-\imath e^{\imath (j-1)\beta}$ by angles which are multiple of $5\beta$.  (See Fig.\ \ref{fig:basis}.)
Combining \eqref{eq:alphabasis} and \eqref{eq:betabasis} we obtain that $F_j$
has the property that $|p-q|\geq r$ if $p,q\in F_j$, $p\neq q$, $j\in\{1,\ldots,5\}$. This simply means that $F_j$ satisfies condition {\rm (i)} in the lemma, $j=1,\ldots,5$. The same holds for $F_{-j} := \{-\overline z\colon z\in F_j\}$, $j=1,\ldots,5$, that is,
\begin{equation}\label{eq:Fx2}
F_{-j} = \{-\imath e^{-\imath (j-1)\beta} e^{-\imath 5(k-1)\beta}\colon k\in\n,\; ((j-1)+5(k-1))\beta\leq\pi\}.
\end{equation}
\begin{figure}[ht]
    \begin{center}
 \scalebox{0.3}{\includegraphics{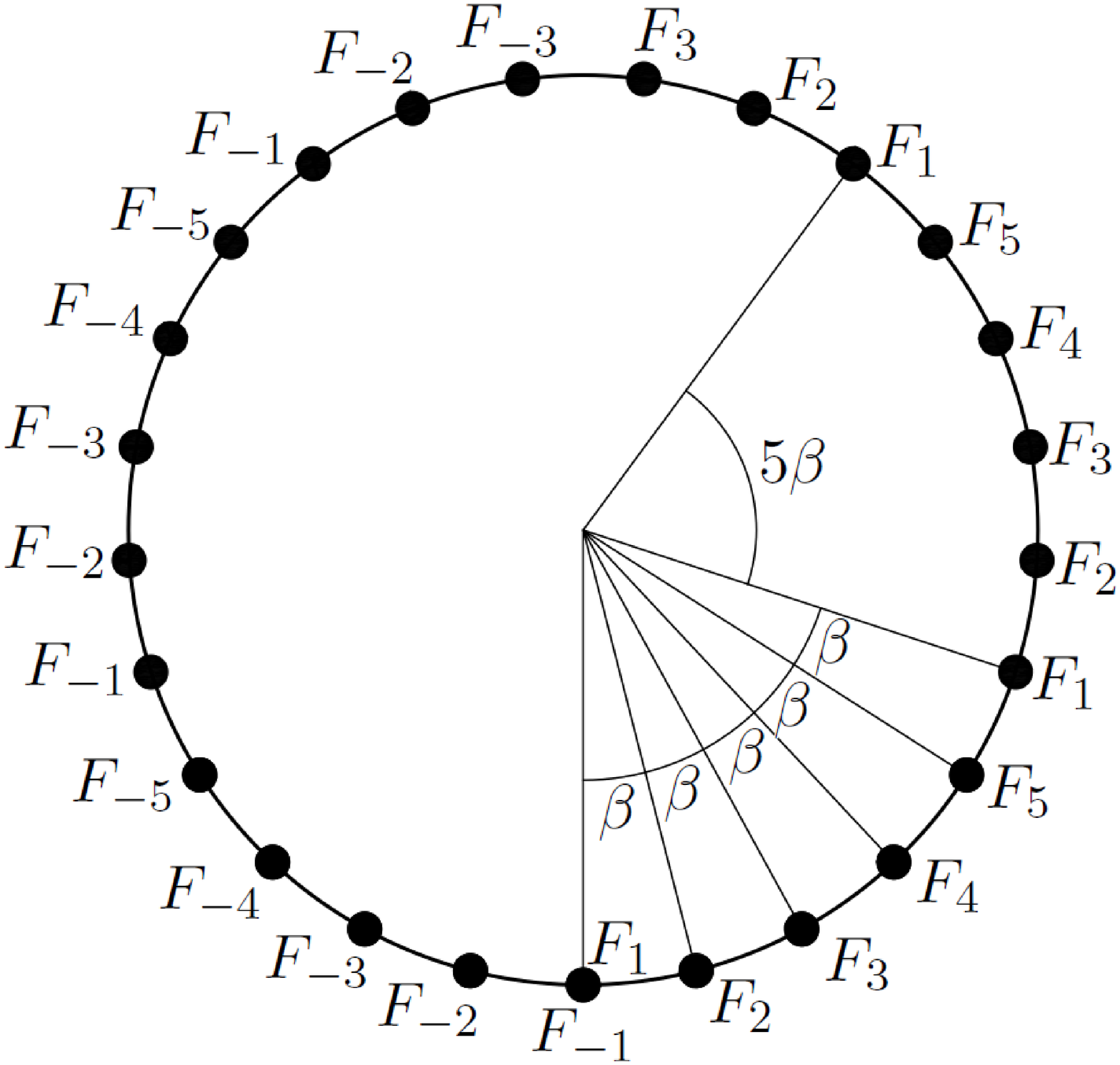}}
     \end{center}
\caption{Basis of the induction for $r=3/4$.}
\label{fig:basis}
\end{figure}

To finish, it suffices to show that $F_1,\ldots,F_5,F_{-1},\ldots, F_{-5}$ also satisfy {\rm (ii)}. 
For that pick $p=-\imath e^{\imath t} \in\s^1$, $t\in[-\pi,\pi]$. Without loss of generality we assume that $t\in[0,\pi]$; otherwise we reason in a symmetric way. It is clear that there exist $j_0\in\{1,\ldots,5\}$ and $k_0\in\n$ such that $((j_0-1)+5(k_0-1))\beta\leq\pi$ and 
$|t-((j_0-1)+5(k_0-1))\beta|\leq\beta$. Thus, since $e^{\imath ((j_0-1)+5(k_0-1))\beta}\in F_{j_0}$, \eqref{eq:distangle2} implies that  
\[
\dist\Big(p,\bigcup_{j=1}^5 (F_j\cup F_{-j}) \Big)  \leq  \dist(p,F_{j_0})
\leq \Big|p-pe^{\imath \big(((j_0-1)+5(k_0-1))\beta-t\big)}\Big| \leq \cgot_1r,
\]
which proves {\rm (ii)} and concludes the basis of the induction.
 
For the inductive step fix $n\in\n$, $n\geq 2$, assume that the lemma holds for $n-1$, and let us prove it for $n$. 

We begin with some preparations. Given $t\in[-\frac{\pi}2,\frac{\pi}2]$ we denote by $\Pi_t$ the affine hyperplane of $\r^{n+1}$ given by 
\[
  \Pi_t:=\{(x_1,\ldots,x_{n+1})\in\r^{n+1}\colon x_{n+1}=\sin t\}.
\]
Notice that 
\begin{equation}\label{eq:snt}
  \s^n\cap \Pi_t=\cos(t)\s^{n-1}\times\{\sin t\}\quad \forall t\in\big[-\frac{\pi}2,\frac{\pi}2\big].
\end{equation}
We adopt the convention $0\s^{n-1}=\{0\}\subset  \r^n$. Obviously, $\s^n=\bigcup_{t\in[-\frac{\pi}2,\frac{\pi}2]}(\s^n\cap\Pi_t)$.

The idea for the proof is similar to the one in the basis of the induction. In this case the role of the points $-\imath e^{\pm\imath (j-1)\beta}$ will be played by finite subsets $A_j^t$, $j=1,\ldots,\mgot_{n-1}$, of the $(n-1)$-dimensional sphere $\s^n\cap\Pi_t$, $t\in]-\frac{\pi}2,\frac{\pi}2[$, which will be provided by the inductive hypothesis; take into account Claim \ref{cla:t}. (See Properties {\rm (P1)}, {\rm (P2)}, and \eqref{eq:Ajt} below.) Likewise, the role of the set $F_j$ will be played by the union of a finite family of sets $A_j^t$ where $j$ is fixed and $t$ moves in sufficiently far heights (compare \eqref{eq:Fjk} below with \eqref{eq:Fx} and \eqref{eq:Fx2}, and Properties {\rm (P3)} and {\rm (P4)} below with \eqref{eq:alphabasis} and \eqref{eq:distangle2}).

So, pick any number $\cgot_n$ with
\begin{equation}\label{eq:cn}
0<\cgot_{n-1}<\cgot_n<\frac12.
\end{equation}

Consider the function $f\colon ]0,+\infty[\to\r$ given by
\begin{equation}\label{eq:f}
f(x)=\frac{\displaystyle \arcsin\big( \min\big\{\frac{x}2,1\big\} \big)}{\displaystyle \arcsin\Big( \min\Big\{\frac{(\cgot_n-\cgot_{n-1})x}2,1 \Big\}\Big)},\quad x>0.
\end{equation}
Observe that $f$ is continuous, positive, $\lim_{x\to 0}f(x)=1/(\cgot_n-\cgot_{n-1})>2$, and $\lim_{x\to+\infty}f(x)=1$, hence $f$ is a bounded function. Set 
\begin{equation}\label{eq:mu}
\mu:= {\rm E}\big(\sup_{x> 0}f(x)\big)\geq 2,
\end{equation}
where ${\rm E}(\cdot)$ means integer part. Obviously $\mu$ only depends on $\cgot_{n-1}$ and $\cgot_n$. Set 
\[
      \mgot_n:=(\mu+1) \mgot_{n-1},
\]
and let us check that the numbers $\mgot_n\in\n$ and $\cgot_n\in]0,1/2[$ satisfy the conclusion of the lemma. For that, choose $r>0$ and let us furnish subsets $F_1,\ldots,F_{\mgot_n}$ of $\s^n$ meeting conditions {\rm (i)} and {\rm (ii)} in the statement of the lemma.

For $t$ in the open interval $]-\frac{\pi}2,\frac{\pi}2[$ and in view of \eqref{eq:snt}, Claim \ref{cla:t} provides subsets $A_1^t,\ldots, A_{\mgot_{n-1}}^t$ of $\s^n\cap\Pi_t$  satisfying the thesis of the lemma for the real number $r>0$. 
 That is:
\begin{enumerate}
\item[\rm ({P}1)]  $A_j^t$ is finite and $|p-q|\geq r$ for all $p, q\in A_j^t$, $p\neq q$, $j=1,\ldots,\mgot_{n-1}$. 
\item[\rm ({P}2)] $A^t:=\bigcup_{j=1}^{\mgot_{n-1}} A_j^t\neq\emptyset$ and $\dist(p,A^t)\leq \cgot_{n-1} r$ for all $p\in \s^n\cap\Pi_t$. 
\end{enumerate}
We also set 
\begin{equation}\label{eq:Ajt}
A_j^{\pm\frac{\pi}2}:=\{(0,\ldots,0,\pm1)\}\subset\s^n, \quad j\in\{1,\ldots,\mgot_{n-1}\}.\end{equation}

To finish the proof, we will distribute a suitable subset of $\bigcup_{(t,j)\in [-\frac{\pi}2,\frac{\pi}2]\times\{1,\ldots,\mgot_{n-1}\}} A_j^t\subset\s^n$ into $\mgot_n$ subsets of $\s^n$ satisfying conditions {\rm (i)} and {\rm (ii)}. 

Observe first that, by basic trigonometry,
\begin{equation}\label{eq:distst}
\dist(p,\s^n\cap\Pi_t)= 2 \sin \Big(\frac{|t-s|}2\Big)\quad \forall p\in \s^n\cap\Pi_s,\; t,s\in\big[-\frac{\pi}2,\frac{\pi}2\big].
\end{equation}
Set
\begin{equation}\label{eq:ab}
\alpha:=2\arcsin\Big( \min\Big\{\frac{r}2,1\Big\} \Big),\quad \beta:=2\arcsin\Big( \min\Big\{\frac{(\cgot_n-\cgot_{n-1})r}2,1 \Big\}\Big).
\end{equation}
In view of \eqref{eq:cn}, we have $0<\beta\leq\alpha\leq\pi$. Moreover, since the sinus is increasing in $[0,\frac{\pi}2]$, \eqref{eq:distst} ensures that:
\begin{enumerate}
\item[\rm ({P}3)] $\dist(\s^n\cap\Pi_s,\s^n\cap\Pi_t)\geq \min\{r,2\}$ for all $t,s\in [-\frac{\pi}2,\frac{\pi}2]$ with $|t-s|\geq \alpha$.
\item[\rm ({P}4)] $\dist(p,\s^n\cap\Pi_t)\leq \min\{(\cgot_n-\cgot_{n-1})r,2\}$ for all $p\in \s^n\cap\Pi_s$, for all $t,s\in[-\frac{\pi}2,\frac{\pi}2]$ with $|t-s|\leq\beta$.
\end{enumerate}
Further, $\alpha/\beta=f(r)$ where $f$ is the function \eqref{eq:f}, and so \eqref{eq:mu} gives that
\begin{equation}\label{eq:mu+1}
(\mu+1)\beta>\alpha.
\end{equation}

Denote by $I:=\{1,\ldots,\mgot_{n-1}\}\times \{0,\ldots,\mu\}$. For all $k\in \{0,\ldots,\mu\}$ call $I_k:=\{l\in\z\colon 0\leq k+l(\mu+1)\leq\pi/\beta\}$, and notice that $I_k$ is a (possibly empty) finite set; see \eqref{eq:mu}.  For $(j,k)\in I$, set
\begin{equation}\label{eq:Fjk}
F_{j,k}:=\bigcup_{l\in I_k} A_j^{-\frac{\pi}2+(k+l(\mu+1))\beta}.
\end{equation}
(See Figure \ref{fig:step}.)
\begin{figure}[ht]
    \begin{center}
 \scalebox{0.3}{\includegraphics{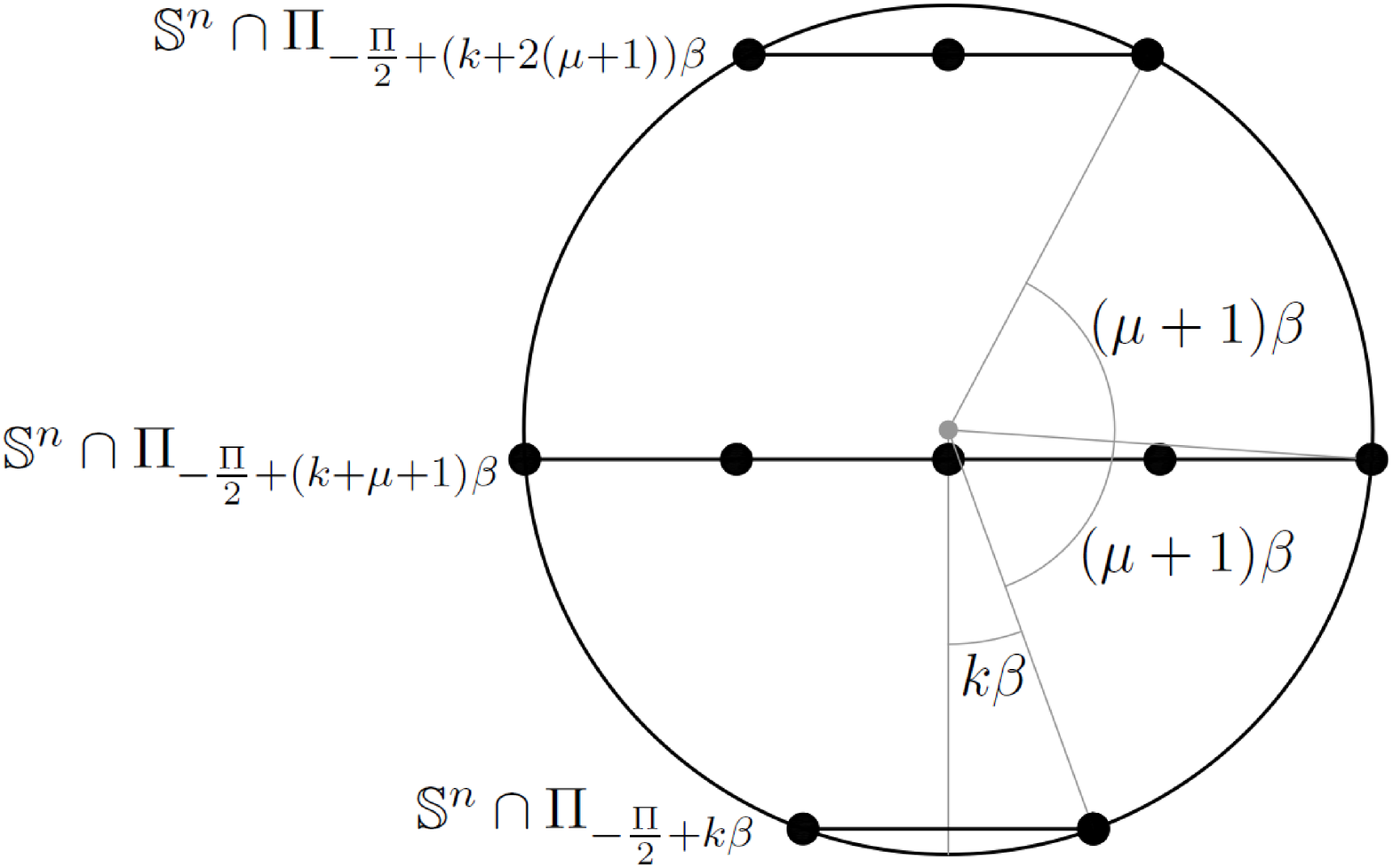}}
     \end{center}
\caption{The subset $F_{j,k}\subset \s^n$.}
\label{fig:step}
\end{figure}
To conclude the proof it suffices to check that the $\mgot_n=(\mu+1) \mgot_{n-1}$  subsets of $\s^n$, 
$F_{j,k}$, $(j,k)\in I$, satisfy Lemma \ref{lem:points}-{\rm (i)},{\rm (ii)}. 

Pick $(j,k)\in I$. Since $I_k$ is finite, {\rm (P1)} and \eqref{eq:Ajt} ensure that the (possibly empty) set $F_{j,k}$ is finite. Suppose that $F_{j,k}\neq\emptyset$ and
choose a pair of points $p,q\in F_{j,k}$. If $p,q\in  A_j^{-\frac{\pi}2+(k+l(\mu+1))\beta}$ for some $l\in I_k$, {\rm (P1)} and \eqref{eq:Ajt} ensure that either $p=q$ or $|p-q|\geq r$. Assume now that $p\in  A_j^{-\frac{\pi}2+(k+l_1(\mu+1))\beta}$ and $q\in  A_j^{-\frac{\pi}2+(k+l_2(\mu+1))\beta}$ with $l_1,l_2\in I_k$, $l_1\neq l_2$. It follows that 
\begin{eqnarray*}
\pi & \geq & \big|-\frac{\pi}2+(k+l_1(\mu+1))\beta-\big(-\frac{\pi}2+(k+l_2(\mu+1))\beta\big)\big|
\\
 & = &  |l_1-l_2|(\mu+1) \beta \; \geq \; (\mu+1) \beta \; >\; \alpha;
\end{eqnarray*}
take into account  \eqref{eq:mu+1}. This and \eqref{eq:ab} imply $r<2$, hence {\rm (P3)} guarantees that $|p-q|\geq\min\{r,2\}=r$. Condition {\rm (i)} follows. 

To check {\rm (ii)} set $F:=\bigcup_{(j,k)\in I} F_{j,k}$. Since $\emptyset\neq A_j^{-\frac{\pi}2}\subset F_{j,0}\subset F$,  $F\neq\emptyset$. Pick a point $p\in\s^n$ and let $s\in[-\frac{\pi}2,\frac{\pi}2]$ be the unique number satisfying $p\in\s^n\cap\Pi_s$. Since $0<\beta\leq\pi$, there exist  $k_0\in\{0,\ldots,\mu\}$ and  $l_0\in I_{k_0}$  such that $|t_0-s|\leq\beta$, where  $t_0:=-\frac{\pi}2+(k_0+l_0(\mu+1))\beta\in[-\frac{\pi}2,\frac{\pi}2]$. Property {\rm (P4)} provides $q\in \s^n\cap\Pi_{t_0}$ such that $|p-q|\leq \min\{(\cgot_n-\cgot_{n-1})r,2\}$. Together with {\rm (P2)} and \eqref{eq:Ajt}, we obtain that 
\begin{eqnarray*}
\dist\big(p,\bigcup_{j=1}^{\mgot_{n-1}} A_j^{t_0}\big) & \leq & |p-q|+\dist\big(q,\bigcup_{j=1}^{\mgot_{n-1}} A_j^{t_0}\big)
\\
& \leq & \min\{(\cgot_n-\cgot_{n-1})r,2\}+\cgot_{n-1}r \; \leq \; \cgot_nr.
\end{eqnarray*}
Since $\bigcup_{j=1}^{\mgot_{n-1}} A_j^{t_0} \subset \bigcup_{j=1}^{\mgot_{n-1}} F_{j,k_0}\subset F$, the above inequality proves {\rm (ii)}. This completes the proof.
\end{proof}


With Lemma \ref{lem:points} in hand, we may find finite tidy collections of tangent balls in a spherical shell, which are suitable for our purposes.

\begin{lemma}\label{lem:r1r2}
Given $n\in \n$, there exists a number $\agot_n>0$ such that the following holds:

For any $0<r_1<r_2<1$ there is a finite tidy collection  $\Fgot$ of tangent balls in $r_2 \b$, contained in $r_2\b\setminus r_1\overline \b$, such that 
if $\gamma\colon [0,1]\to r_2\overline \b\setminus r_1\b$ is a path such that $|\gamma(0)|=r_1$, $|\gamma(1)|=r_2$, and $\gamma([0,1])\cap|\Fgot|=\emptyset$, then the length of $\pgot\circ \gamma$ is at least  $\displaystyle \agot_n  \frac{\sqrt{r_2-r_1} }{\sqrt{r_2}}$. In particular, the length of such $\gamma$ is at least  $\displaystyle \agot_n  \frac{r_1\sqrt{r_2-r_1} }{\sqrt{r_2}}$.
\end{lemma}
\begin{proof}
By Lemma \ref{lem:points} there are $\mgot_n\in\n$, $\mgot_n\geq 2$, and $\cgot_n\in\r$, $0<\cgot_n<1/2$, such that 
given $r>0$ there are finite sets $F_1,\ldots,F_{\mgot_n}\subset \s^n$ satisfying: 
\begin{enumerate}[{\rm (i)}]
\item $|p-q|\geq r$ for all $p$, $q\in F_j$, $p\neq q$, $1\leq j \leq \mgot_n$.
\item If $F=\bigcup_{j=1}^{\mgot_n} F_j$ then for every $p\in \s^n$ we have $\dist (p,F)\leq \cgot_n r$.
\end{enumerate}
Fix $r>0$ and let $F_j$, $1\leq j\leq \mgot_n$, be as above.

By {\rm (i)}, given $p\in \s^n$ there is a $q\in F$ such that $|q-p|\leq \cgot_n r$. So, if $y\in \s^n$, $|y-p|\leq (1/2-\cgot_n) r$, then $|y-q|\leq |y-p|+|p-q|\leq  (1/2-\cgot_n) r +  \cgot_n r=\frac{r}{2}$. Thus, for every $p\in \s^n$ there are $j$, $1\leq j\leq \mgot_n$, and $q\in F_j$ such that 
\begin{equation}\label{eq:1st}
\big(p+(1/2-\cgot_n)r \overline \b\big)\cap \s^n\subset \big( q+\frac{r}{2} \overline \b \big) \cap \s^n.
\end{equation}
For each $p\in \s^n$ denote by $\oB(p)$ the tangent ball $\oB((1-\delta)p,\eta)$, where $0<\delta<1$ and $\eta>0$ are chosen so that this ball is attached to $\s^n$ along $(p+\frac{r}{2} \s^n)\cap \s^n$, that is to say, the boundary
\[
b\big( \oB(p) \big) = \big( p+\frac{r}{2} \s^n \big) \cap \s^n;
\]
see Figure \ref{fig:collection}.
This simply means that $(1-\delta)^2+\eta^2=1$ and $\eta^2+\delta^2=(\frac{r}{2})^2,$ which implies that $\frac{r}{2}=\sqrt{2 \delta}$.

Note that
\[
\text{$\oB(p)\cap \oB(q)=\emptyset$ whenever $p,$ $q\in \s^n$, $|p-q|\geq r$,}
\]
and that the projection to $\s^n$
\[
\pgot(\oB(p))=\big(p+\frac{r}{2} \overline \b\big)\cap \s^n.
\]
Now, let $0<r_1<r_2<1$. Divide the interval $[r_1,r_2]$ into $\mgot_n+1$ equal pieces of length 
\[
\omega=\frac{r_2-r_1}{\mgot_n+1}
\]
and let $s_j=r_1+j \omega$, $0\leq j\leq \mgot_n+1$, so that 
\[
r_1=s_0<s_1<\cdots<s_{\mgot_n}<s_{\mgot_n+1}=r_2.
\]
We now describe how to get our tidy collection of tangent balls.

Given $j$, $1\leq j\leq \mgot_n$, let $G_j$ be the collection $\{\oB(p)\colon p\in F_j\}$. This is a collection of pairwise disjoint tangent balls whose boundaries are contained in $\s^n$ and whose centers are on $(1-\delta) \s^n$. Now form $s_j G_j$ by multiplying each ball in $G_j$ by $s_j$ and thus pulling it inside $\b$; so $s_j G_j=\{s_j \oB(p)\colon p\in F_j\}$. Observe that $s_j G_j$ is a collection of pairwise disjoint tangent balls with centers on $s_j (1-\delta) \s^n$ and with boundaries contained in $s_j \s^n$.

Consider now the collection $\Fgot$ of all tangent balls obtained in this way (see Figure \ref{fig:collection}), that is:
\[
\Fgot=\bigcup_{j=1}^{\mgot_n} s_j G_j.
\]
\begin{figure}[ht]
    \begin{center}
 \scalebox{0.4}{\includegraphics{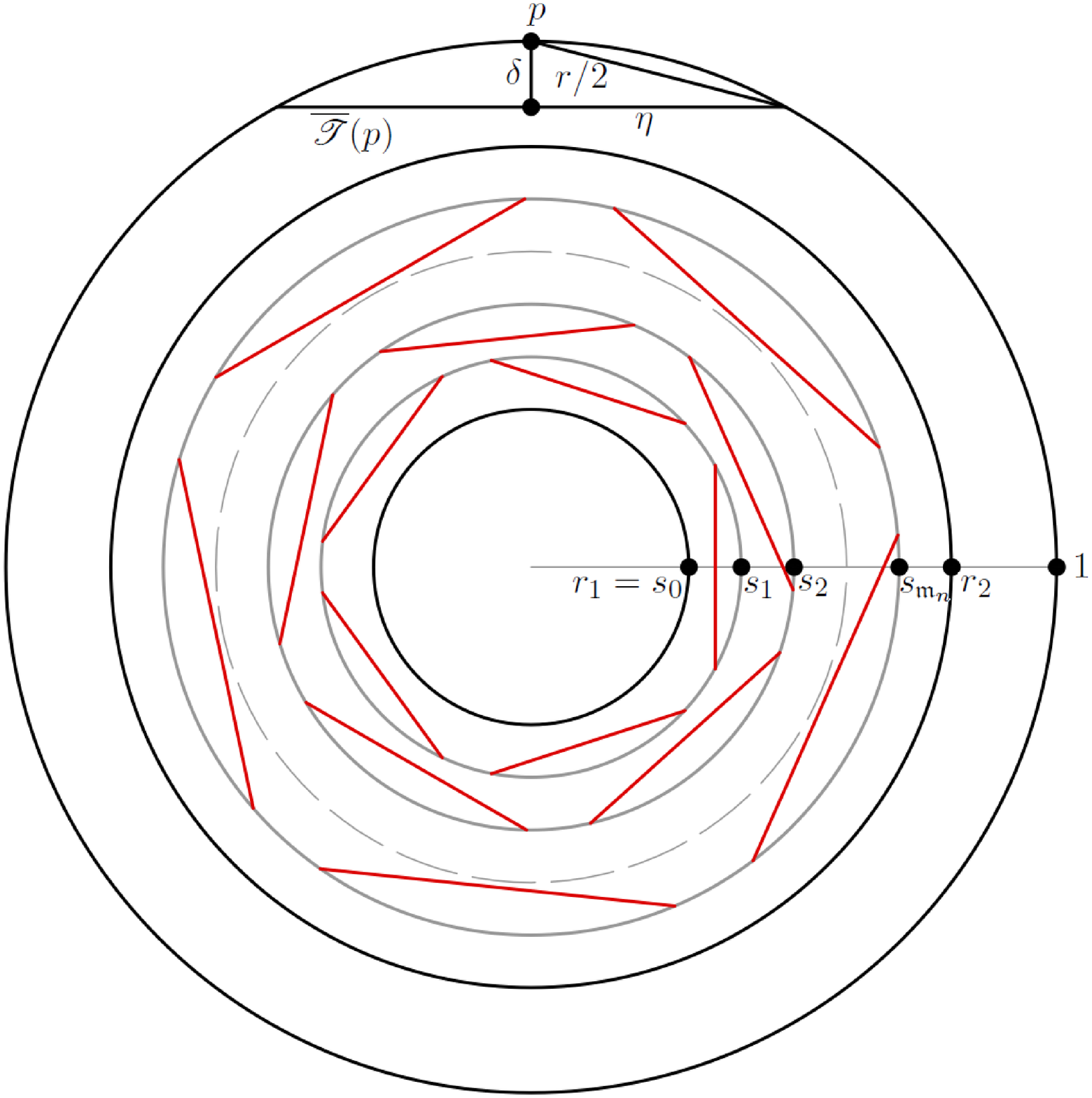}}
     \end{center}
\caption{The collection $\Fgot=\bigcup_{j=1}^{\mgot_n} s_j G_j$.}
\label{fig:collection}
\end{figure}
The familly $\Fgot$ will be a tidy collection of tangent balls in $r_2  \b$ contained in $r_2 \b\setminus r_1\overline \b$ provided that for each $j$, $2\leq j\leq \mgot_n$, the sphere containing the centers of the balls in $s_j G_j$, that is, $(1-\delta) s_j \s^n$, is outside the ball $s_{j-1} \overline \b$ which contains the balls of $s_{j-1} G_{j-1}$, that is, provided that 
\begin{equation}\label{eq:2nd}
s_j (1-\delta)>s_{j-1}\quad  \text{for all $2\leq j\leq \mgot_n$},
\end{equation}
and since we do not want the balls to meet $r_1 \overline\b=s_0\overline \b$, \eqref{eq:2nd} will have to hold for $j=1$ as well.
Now, \eqref{eq:2nd} for $1\leq j\leq \mgot_n$ means that 
\begin{equation} \label{eq:3rd}
\delta<\frac{s_j-s_{j-1}}{s_j}=\frac{\omega}{s_j},\quad 1\leq j\leq \mgot_n.
\end{equation}
Since $s_j\leq r_2$, $ 1\leq j\leq \mgot_n$, \eqref{eq:3rd} certainly holds if we choose 
\[
\delta:=\frac{\omega}{r_2}=\frac{r_2-r_1}{(\mgot_n+1) r_2}.
\]
Recall that $\frac{r}{2}=\sqrt{2\delta}$. Set
\begin{equation} \label{eq:4th}
r:=2\sqrt{ 2 \frac{r_2-r_1}{(\mgot_n+1) r_2}}
\end{equation}
at the outset and obtain $\Fgot$ as above. Then $\Fgot$ is a tidy collection of balls contained in $r_2\overline \b\setminus r_1\b$.

Let us show that $\Fgot$ has the  properties stated in the lemma.

Let $q\in \s^n$ and set $\Omega:=(q+\frac{r}{2} \b)\cap \s^n$. Observe that $\Omega$ is an open subset of $\s^n$ whose closure equals $\pgot(\oB(q))=\pgot(s\oB(q))$ for any $s>0$. Note also that for every $s>0$ the tangent ball $s\oB(q)$ cuts the open cone into two components, that is, $\pgot^{-1}(\Omega)\setminus s\oB(q)$ consists of two components.

Suppose that $\gamma\colon [0,1]\to r_2 \overline \b\setminus r_1 \b$ is a path with $|\gamma(0)|=r_1$, $|\gamma(1)|=r_2$, and $\gamma([0,1])\cap|\Fgot|=\emptyset$. Write $p=\pgot(\gamma(0))$. By \eqref{eq:1st}, there are $j\in \{1,\ldots,\mgot_n\}$ and $q\in F_j$ such that 
$\big(p+(\frac12-\cgot_n)r \overline \b\big)\cap \s^n\subset (q+\frac{r}{2} \overline \b) \cap \s^n$, and so, 
\begin{equation}\label{eq:5th}
\big( p+\frac12 \big( \frac12-\cgot_n \big) r \overline \b \big) \cap \s^n\subset \Omega,
\end{equation}
where $\Omega$ is as above.

Assume for a moment that 
\begin{equation}\label{eq:6th}
{\rm diam}\big(\pgot(\gamma([0,1]))\big)<\frac1{2} \big( \frac1{2}-\cgot_n \big) r.
\end{equation}
By \eqref{eq:5th}, this implies that $\gamma([0,1])\subset \Omega$. On the other hand, $s_j\oB(q)$ is the tangent ball in $\Fgot$ contained in $s_j\b \setminus s_{j-1}\overline \b\subset r_2\b \setminus r_1\overline \b$. Since  $|\gamma(0)|=r_1$ and $|\gamma(1)|=r_2$ it follows that  $\gamma(0)$ and $\gamma(1)$ are in different components of $\Omega\setminus s_j \oB(q)$ and therefore $\gamma([0,1])$ meets $s_j \oB(q)$, a contradiction.
So \eqref{eq:6th} does not hold, which implies that 
\[
{\rm Length}(\pgot\circ \gamma)\geq {\rm diam}(\pgot(\pi([0,1]))\geq \frac1{2} \big( \frac12-\cgot_n \big) r.
\]
By \eqref{eq:4th} it follows that
\[
{\rm Length}(\pgot\circ \gamma)\geq \agot_n \frac{\sqrt{r_2-r_1}}{\sqrt{r_2}},
\]
where $\agot_n=(\frac1{2}-\cgot_n) \frac{\sqrt{2}}{\sqrt{\mgot_n+1}}$. This completes the proof.
\end{proof}


Finally, a finite recursive application of Lemma \ref{lem:r1r2} gives the following extension.

\begin{lemma}\label{lem:r1r2+}
Suppose  $0<r_1<r_2<1$ and let $\rho>0$. There exists a finite tidy collection  $\Fgot$ of tangent balls in $r_2 \b$ enjoying the following properties:
\begin{enumerate}[\rm (i)]
\item $|\Fgot|\cap r_1\overline \b=\emptyset$.
\item If $\gamma\colon [0,1]\to r_2\overline \b\setminus r_1\b $ is a path with $|\gamma(0)|=r_1$,  $|\gamma(1)|=r_2$, and $\gamma([0,1])\cap|\Fgot|=\emptyset$, then the length of $\pgot\circ \gamma$ is at least  $\rho$. In particular, the length of such $\gamma$ is at least  $r_1\rho$.
\end{enumerate}
\end{lemma}
\begin{proof}
Let $\agot_n>0$ be the number given by Lemma \ref{lem:r1r2}. Pick $N\in\n$, 
\begin{equation}\label{eq:N}
N>\frac{r_2\rho^2}{\agot_n^2(r_2-r_1)}.
\end{equation}
Set
\[
s_j:=r_1+j\frac{r_2-r_1}{N},\quad j=1,\ldots,N.
\]
Obviously, $r_1=s_0<s_1<\cdots< s_N=r_2$. For each $j$, $1\leq j\leq N$, apply Lemma \ref{lem:r1r2} with $s_{j-1}$, $s_j$ in place of $r_1$, $r_2$ to get a finite tidy  collection  $\Fgot_j$ of tangent balls in $s_j \b$ having the following properties:
\begin{enumerate}
\item[\rm (i$_j$)] $|\Fgot_j|\cap s_{j-1}\overline \b=\emptyset$.
\item[\rm (ii$_j$)] If $\gamma\colon [0,1]\to s_j\overline \b\setminus s_{j-1}\b$ is a path with $|\gamma(0)|=s_{j-1}$, $|\gamma(1)|=s_j$, and $\gamma([0,1])\cap|\Fgot_j|=\emptyset$, then the length of $\pgot \circ \gamma$ is at least  $\displaystyle \agot_n  \frac{\sqrt{s_j-s_{j-1}} }{\sqrt{s_j}}$. 
\end{enumerate}

Set $\Fgot:=\bigcup_{j=1}^N\Fgot_j$. It follows that $\Fgot$ is a tidy finite collection  of tangent balls in $r_2 \b$ satisfying {\rm (i)}. Furthermore, given a path $\gamma$ as in  {\rm (ii)}, $\gamma$ contains subpaths $\gamma_1,\ldots,\gamma_N$ with $\gamma_j\subset s_j\overline\b\setminus s_{j-1}\b$ connecting $s_{j-1}\s^n$ and $s_j\s^n$ for all $j$. Therefore,
\begin{eqnarray*}
{\rm Length}(\pgot \circ\gamma) & \geq & \sum_{j=1}^N{\rm Length}(\pgot \circ\gamma_j)
\\
& \stackrel{\text{\rm (ii$_j$)}}\geq & \agot_n \sum_{j=1}^N \frac{\sqrt{s_j-s_{j-1}} }{\sqrt{s_j}}
\\
& = & \sum_{j=1}^N \frac{\agot_n\sqrt{r_2-r_1}}{\sqrt{r_1+\frac{j}{N}(r_2-r_1)}}\frac1{\sqrt{N}}
\\
& \geq & \sum_{j=1}^N \frac{\agot_n\sqrt{r_2-r_1}}{\sqrt{r_2}}\frac1{\sqrt{N}}
\\
& = & \frac{\agot_n\sqrt{r_2-r_1}}{\sqrt{r_2}}\sqrt{N} \; \stackrel{\text{\eqref{eq:N}}}{>} \rho.
\end{eqnarray*}
This proves {\rm (ii)} and completes the proof of the lemma.
\end{proof}

We now state and prove the main result in this section, which is a more general version of Theorem \ref{th:infinity}.

\begin{theorem}\label{th:infinity2}
Let $0<s_0<s_1<\cdots<\lim_{j\to+\infty} s_j=1$ and $0=\rho_0<\rho_1<\cdots <\lim_{j\to+\infty} \rho_j=+\infty$ be sequences of real numbers. There exists a tidy collection $\Fgot$ of tangent balls in $\b$ satisfying the following properties:
\begin{enumerate}[\rm (i)]
\item $|\Fgot|\cap s_1\overline \b=\emptyset$.
\item If $\gamma\colon [0,1]\to \b$ is a path with $|\gamma(0)|=s_{j-1}$, $|\gamma(1)|=s_j$, and $\gamma([0,1])\cap|\Fgot|=\emptyset$, then the length of $\pgot\circ\gamma$ is at least  $\rho_j$. In particular, the length of such $\gamma$ is at least  $s_{j-1}\rho_j$.
\item If $\gamma\colon [0,1[\to \b$ is a path such that $\lim_{t\to1}|\gamma(t)|=1$ and whose range intersects at most finitely many balls in $\Fgot$, then $\pgot\circ \gamma$ and $\gamma$ have infinite length.
\end{enumerate}

Thus, every closed submanifold of $\b$ missing $|\Fgot|$ is complete.
\end{theorem}
\begin{proof}
For each  $j\in \n$, Lemma \ref{lem:r1r2+} applied to $s_{j-1}$, $s_j$ in place of $r_1,$ $r_2$, furnishes a tidy finite collection  $\Fgot_j$ of tangent balls in $s_j \b$ satisfying that:
\begin{enumerate}[\rm (i$_j$)]
\item $|\Fgot_j|\cap s_{j-1}\overline \b=\emptyset$.
\item If $\gamma\colon [0,1]\to s_j\overline \b\setminus s_{j-1}\b $ is a path with $|\gamma(0)|=s_{j-1}$, $|\gamma(1)|=s_j$, and $\gamma([0,1])\cap|\Fgot_j|=\emptyset$, then the length of $\pgot\circ \gamma$ is at least  $\rho_j$. In particular, the length of such $\gamma$ is at least  $s_{j-1}\rho_j$.
\end{enumerate}

Set $\Fgot:=\bigcup_{j\in \n} \Fgot_j$. By {\rm (i$_j$)}, $j\in \n$, $\Fgot$ is a tidy collection of tangent balls in $\b$ satisfying {\rm (i)}, whereas {\rm (ii)} clearly follows from {\rm (ii$_j$)}, $j\in \n$.

To finish, assume that $\gamma\colon [0,1[\to \b$ is a proper path intersecting at most finitely many balls in $\Fgot$. Since $\Fgot$ is tidy in $\b$, there exists $j_0\in \n$ such that  $\gamma([0,1[)\cap |\Fgot_j|=\emptyset$ for all $j\geq j_0$. Up to enlarging $j_0$ if necessary, we infer that $\gamma$ contains a subpath $\gamma_j$ with range in $s_j\overline \b\setminus (|\Fgot_j|\cup s_{j-1} \b)$  connecting $s_{j-1} \s^n$ and $s_j\s^n$, $j\geq j_0$. 
Item  {\rm (ii)} ensures that 
\[
{\rm Length}(\pgot\circ \gamma)\geq \sum_{j\geq j_0} {\rm Length}(\pgot\circ \gamma_j)\geq \sum_{j\geq j_0}  \rho_j=+\infty
\]
and so 
\[
{\rm Length}(\gamma)\geq \sum_{j\geq j_0} {\rm Length}(\gamma_j)\geq  \sum_{j\geq j_0} s_{j-1} \rho_j\geq s_{j_0-1} \sum_{j\geq j_0} \rho_j=+\infty.
\]
This proves {\rm (iii)} and completes the proof.
\end{proof}


\section{The Avoiding Obstacles Theorem}\label{sec:disc}

This section is devoted to the proof of Theorem \ref{th:disc}. Throughout the section we fix $N\in\n$ and write $\bc$  for the unit ball  $\bc_{N+1}\subset \c^{N+1}$. 
Numbers in $\c$ will be denoted by roman letters, typically $z$ and $w$, whereas elements of $\c^N$ or $\c^{N+1}$ will be denoted by greek letters such as $\zeta$ and $\xi$.

As a preliminary step to the proof of Theorem \ref{th:disc} we shall prove the following.

\begin{lemma}\label{lem:shear1}
Let $D$ and $E$ be compact sets in $\c^{N+1}$ and assume that $E$ is convex and $D$ is contained in an affine real hyperplane $L\subset\c^{N+1}$ which does not intersect $E$. Let $Z\subset \c^{N+1}$ be a 
closed complex hypersurface. Given $\epsilon>0$ there exists a holomorphic automorphism  $\Phi\colon\c^{N+1}\to\c^{N+1}$ 
such that:
\begin{enumerate}[\rm (i)]
\item $\Phi(Z)\cap D=\emptyset$.
\item $|\Phi(\zeta)-\zeta|<\epsilon$ for all $\zeta\in E$.
\end{enumerate}
\end{lemma}
The authors wish to thank the referee for pointing out how to prove this lemma which replaces a similar one proved in a previous version of the paper under the extra assumption that the hypersurface $Z$ is algebraic.
\begin{proof}
Let $\pi_1\colon \c^{N+1}\to\c$ denote the orthogonal projection into the first component given by $\pi_1(z_1,\ldots,z_{N+1})=z_1$.
By an affine complex change of coordinates we may assume that 
\[
    L=\{(z_1,\ldots,z_{N+1})\in\c^{N+1}\colon \Re(z_1)=0\},
\]
where $\Re(\cdot)$ means real part, and $E$ is contained in the half space $\{\Re(z_1)<0\}$. 
Since $D\subset L$ is compact, there exist a compact segment $D'$ in the real line $\{z\in\c\colon \Re(z)=0\}=\pi_1(L)$ and a real number $\lambda>0$ such that
\begin{equation}\label{eq:D'd}
      D\subset D'\times (\lambda\bc_N)\subset\c\times\c^N.
\end{equation}
Further, by dimension reasons we may assume that $Z\cap (D'\times\{0\})=\emptyset$ (here $0\in\c^N$; this can be achieved, for example, by an arbitrarily small translation of the hypersurface $Z$). Thus, since $Z$ is closed and $D'$ is compact, there is $\eta>0$ such that
\begin{equation}\label{eq:D'eta}
      Z\cap( D'\times (\eta\bc_N))=\emptyset.
\end{equation}
On the other hand, $E':=\pi_1(E)\subset \c$ is a compact convex subset of $\{z\in\c\colon \Re(z)<0\}$. 

Pick a small $\tau>0$, which will be specified later, and choose a holomorphic function $\psi\colon\c\to\c$ such that 
\begin{equation}\label{eq:<}
|\psi(z)|<\tau\quad \text{for all $z\in E'=\pi_1(E)$}
\end{equation}
and 
\begin{equation}\label{eq:>}
\Re(\psi(z))>1/\tau\quad \text{for all $z\in D'\supset \pi_1(D)$}.
\end{equation}
Such exists by the classical Runge approximation theorem; observe that $D'$ and $E'$ are connected, simply-connected, disjoint compact subsets of $\c$, and hence $D'\cup E'$ is Runge in $\c$. We claim that, if $\tau>0$ is chosen small enough, then the holomorphic automorphism $\Phi\colon\c^{N+1}=\c\times\c^N\to\c^{N+1}=\c\times\c^N$ given by
\[
    \Phi(z,\xi)=(z, e^{\psi(z)}\xi),\quad z\in\c,\; \xi\in\c^N,
\]
satisfies the conclusion of the lemma. Indeed, for $(z,\xi)\in E$ we have 
\[
   |\Phi(z,\xi)-(z,\xi)|=|(1-e^{\psi(z)})\xi|\leq |1-e^{\psi(z)}|\max\{|\zeta|\colon \zeta\in E\}<\epsilon,  
\]
where the last inequality is ensured by \eqref{eq:<} provided that $\tau>0$ is sufficiently small; recall that $E$ is compact. This implies {\rm (ii)}. On the other hand, for $z\in D'$ we have
\[
    \Phi(\{z\}\times (\eta\bc_N))=\{z\}\times (e^{\psi(z)}\eta\bc_N)\supset \{z\}\times (\lambda\bc_N),
\]
where the last inclusion is guaranteed by \eqref{eq:>} provided that $\tau>0$ is sufficiently small. Thus, in view of \eqref{eq:D'd},
\[
    D\subset D'\times (\lambda\bc_N)\subset \Phi(D'\times (\eta\bc_N)).
\]
Since \eqref{eq:D'eta} ensures that $\Phi(Z)\cap \Phi(D'\times (\eta\bc_N))=\emptyset$, the above inclusion guarantees {\rm (i)}. This concludes the proof. 
\end{proof}

We now prove Theorem \ref{th:disc} by a recursive application of Lemma \ref{lem:shear1}.

\begin{proof}[Proof of Theorem \ref{th:disc}]
Let $\Fgot$ be a tidy collection of tangent balls in the unit ball $\bc\subset\c^{N+1}$ (see Def.\ \ref{def:balls} and \ref{def:tidy}) and fix numbers $\epsilon>0$ and $\lambda_0>0$ such that $|\Fgot|\cap \lambda_0 \overline\bc=\emptyset$. Let $Z\subset\c^{N+1}$ be a  closed complex hypersurface such that $Z\cap\lambda_0\overline\bc\neq\emptyset$.

 Pick $\lambda_0<\lambda_1<1$ such that $|\Fgot|\cap \lambda_1 \overline\bc=\emptyset$.
With no loss of generality we may assume that the collection $\Fgot$ is infinite; otherwise we simply replace $\Fgot$ by any infinite tidy collection in $\bc$ containing it and disjoint from $\lambda_1\overline \bc$. Set $T_0=\emptyset$. Since $\Fgot$ is tidy, it is clear that we may write $\Fgot=\{T_j\colon j\in \n\}$ so that  there exists an exhaustion of $\bc$ by smoothly bounded, strictly convex domains
\[
E_{0}:=\lambda_0 \bc \Subset E_1:= \lambda_1 \bc \Subset E_2 \Subset \ldots \Subset \bigcup_{j\in \z_+} E_j=\bc,
\]
 such that 
\begin{equation}\label{eq:Ej}
\bigcup_{i=0}^{j-1} T_i\subset E_j\quad  \text{and}\quad \overline E_j\cap (\bigcup_{i> j-1} T_i)=\emptyset,\quad j\in \n.
\end{equation} 
Recall the short discussion which follows Theorem \ref{th:disc} in the introduction.

Set $\epsilon_{0}=\epsilon$ and $\Psi_0={\rm Id}\colon \c^{N+1}\to \c^{N+1}$. We shall inductively use Lemma \ref{lem:shear1} to find a sequence $\{\Phi_j\}_{j\in\n}$ of holomorphic automorphisms of $\c^{N+1}$ and a sequence $\{\epsilon_j\}_{j\in\n}$ such that if
\[
\Psi_j:=\Phi_j\circ\cdots\circ \Phi_1,\quad j\in \n,
\]
 then for each $j\in\n$ the following conditions hold:
\begin{enumerate}
\item[\rm (a$_j$)] $\epsilon_j<\epsilon_{j-1}/2$.
\item[\rm (b$_j$)] $\epsilon_j<\dist(\overline E_{j-1},\c^{N+1}\setminus E_j)$.
\item[\rm (c$_j$)] $\epsilon_j<\dist(\overline E_{j},\c^{N+1}\setminus E_{j+1})$.
\item[\rm (d$_j$)] $|\Phi_j(\zeta)  -\zeta|<\epsilon_j$ for all $\zeta\in \overline E_{j}$.
\item[\rm (e$_j$)] $\Psi_j(Z)\cap (\bigcup_{i=0}^{j} T_i)=\emptyset$, and there is an open neighborhood $U_j\subset E_{j}$ of $\bigcup_{i=0}^{j-1} T_i$ (see \eqref{eq:Ej}) such that if $\Theta$ is an automorphism of $\c^{N+1}$ such that $|\Theta(\zeta)-\zeta|<2\epsilon_j$ for all $\zeta\in\overline E_{j}$, then $\Theta(\Psi_{j-1}(Z))\cap U_j=\emptyset$.
\end{enumerate}

Assume for a moment that we have sequences $\{\Phi_j\}_{j\in\n}$ and $\{\epsilon_j\}_{j\in\n}$  satisfying {\rm (a$_j$)}--{\rm (d$_j$)} above, $j\in \n$. By {\rm (b$_j$)} and  {\rm (d$_j$)}, we have that
\begin{equation}\label{eq:inclus0}
\Phi_j(\overline E_{j-1})\subset E_j,\quad j\in \n.
\end{equation}
If we write $L_j=\Psi_j^{-1}(\overline E_j)$, $j\in\n$, we infer from \eqref{eq:inclus0} that  
\begin{equation}\label{eq:eles}
\overline{E}_{0}=\lambda_0 \overline \bc \Subset  L_j\Subset L_{j+1} \quad \text{for all}\quad  j\in\n; 
\end{equation}
for the former inclusion take into account that $\Psi_1=\Phi_1$.
Thus, since {\rm (a$_j$)} ensures that $\sum_{j\in \n} \epsilon_j<+\infty$, \eqref{eq:eles} and properties {\rm (b$_j$)} imply that $\lim_{j\to\infty}\Psi_j=\Psi$ exists uniformly on compacta in 
\[
 \Omega:=\bigcup_{j\in\n} L_j\subset \c^{N+1}
\] and $\Psi$ is a biholomorphic map from $\Omega$ onto $\bigcup_{j\in\n}\overline E_j=\bc$; see \cite[Proposition 4.4.1]{Franc-Book}.

On the other hand, {\rm (c$_j$)} and {\rm (d$_j$)} ensure that
\begin{equation}\label{eq:inclus}
\Phi_j(\overline E_{j})\subset E_{j+1}, \quad j\in \n,
\end{equation}
and so $\Psi_j(\overline E_1)\subset  E_{j+1}$ for all $j\in \n$. Thus, for $\zeta\in  \overline E_1$, {\rm (d$_k$)} and  {\rm (a$_k$)}, $k\leq j$, give
\begin{eqnarray*}
|\Psi_j(\zeta)-\zeta| & \leq & |\Psi_1(\zeta)-\zeta|+  \sum_{k=1}^{j-1} |\Psi_{k+1}(\zeta)-\Psi_{k}(\zeta)|
\\
 & = & |\Phi_1(\zeta)-\zeta|+  \sum_{k=1}^{j-1} |\Phi_{k+1}(\Psi_{k}(\zeta))-\Psi_{k}(\zeta)| \; \leq \; \sum_{k=1}^j \epsilon_j \; < \; 2\epsilon_1.
\end{eqnarray*} 
Since $\lambda_0\overline\bc=\overline E_{0} \Subset \Omega \cap  E_1$ (see \eqref{eq:eles} and the definition of $\Omega$), passing to the limit and taking into account {\rm (a$_1$)} 
we get that $|\Psi(\zeta)-\zeta|\leq2\epsilon_1<\epsilon_{0}=\epsilon$ for all $\zeta\in \lambda_0\overline\bc$. This proves {\rm (i)} in the statement of the theorem.

Assume now that, in addition, the sequences $\{\Phi_j\}_{j\in\n}$ and $\{\epsilon_j\}_{j\in\n}$ satisfy {\rm (e$_j$)}, $j\in\n$. We show that this implies that $\Psi(Z\cap\Omega)$ misses $|\Fgot|$, which proves {\rm (ii)}. Indeed, for each $j\in\n$ and  $k\geq j$ call $\Theta_{j,k}:=\Phi_{k}\circ\cdots\circ\Phi_{j}$. Then $\Theta_{j,k}$ is an automorphism of $\c^{N+1}$. 
By \eqref{eq:inclus},  $\Theta_{j,k}(\overline E_{j})\subset  E_{k+1}$ for all $k\geq  j$. Thus, for each $\zeta \in \overline E_{j}$, {\rm (d$_{i}$)} and {\rm (a$_i$)}, $j\leq  i\leq k$, give 
\begin{eqnarray*}
|\Theta_{j,k}(\zeta)-\zeta| & \leq & |\Theta_{j,j}(\zeta)-\zeta|+ \sum_{i=j+1}^{k} |\Theta_{j,i}(\zeta)-\Theta_{j,i-1}(\zeta)|
\\
 & = & |\Phi_j(\zeta)-\zeta|+ \sum_{i=j+1}^k |\Phi_i(\Theta_{j,i-1}(\zeta))-\Theta_{j,i-1}(\zeta)|
\; \leq \; \sum_{i=j}^k \epsilon_i \; < \; 2\epsilon_j.
\end{eqnarray*}
Therefore, {\rm (e$_j$)} guarantees that $\Theta_{j,k} (\Psi_{j-1}(Z))\cap U_j= \Psi_k(Z)\cap U_j=\emptyset$ for every $k\geq j$, and so $\Psi(Z\cap\Omega)\cap U_j=\emptyset$. Since this holds for every $j\in\n$ and $|\Fgot|\subset \bigcup_{j\in \n} U_j$,   it follows that $\Psi(Z\cap\Omega)\cap|\Fgot|=\emptyset$ as claimed.

It remains to show that there are sequences $\{\Phi_j\}_{j\in\n}$  
and $\{\epsilon_j\}_{j\in\n}$,  
and for each $j\in \n$ an open neighborhood $U_j\subset E_{j}$ of $\bigcup_{i=0}^{j-1} T_i$ (see \eqref{eq:Ej}), such that if $\Psi_j=\Phi_j\circ\Phi_{j-1}\circ\cdots\circ \Phi_1$ then {\rm (a$_j$)}--{\rm (e$_j$)} hold for all $j\in\n$. We proceed by induction. 

For the basis of the induction we choose any number $\epsilon_1>0$ satisfying {\rm (a$_1$)}, {\rm (b$_1$)}, and {\rm (c$_1$)}. By Lemma \ref{lem:shear1} there is a holomorphic automorphism  $\Phi_1$ of $\c^{N+1}$ such that $|\Phi_1(\zeta)-\zeta|<\epsilon_1$ on $\overline E_1$ and $\Phi_1(Z)\cap T_1=\emptyset$. Set $U_1=\emptyset$.  Conditions  {\rm (d$_1$)} and {\rm (e$_1$)} are clear; recall that $\Psi_1=\Phi_1$ and $T_0=\emptyset$.

For the induction step, let $j\geq 2$, assume that we have $\Phi_{j-1}$, $\epsilon_{j-1}$, and $U_{j-1}$ satisfying {\rm (a$_{j-1}$)}--{\rm (e$_{j-1}$)}, and let us provide $\Phi_j$, $\epsilon_j$, and $U_j$ meeting {\rm (a$_j$)}--{\rm (e$_j$)}. 
Fix a number $\epsilon_j>0$ which will be specified later. Assume that $\epsilon_j$ satisfies {\rm (a$_j$)}, {\rm (b$_j$)}, and {\rm (c$_j$)}. 
Thus, Lemma \ref{lem:shear1} furnishes a holomorphic automorphism $\Phi_j$ of $\c^{N+1}$ satisfying {\rm (d$_j$)} and such that $\Psi_j(Z)=\Phi_j(\Psi_{j-1}(Z))$ misses $T_{j}$. Furthermore, since {\rm (e$_{j-1}$)} and \eqref{eq:Ej} ensure that $\Psi_{j-1}(Z)$ misses $\bigcup_{i=0}^{j-1} T_i =E_j\cap|\Fgot|$, we may guarantee that $\Psi_j(Z)$ misses $\bigcup_{i=0}^{j} T_i$ provided that $\epsilon_j$ is chosen small enough. This proves the former assertion in {\rm (e$_j$)}. For the latter one, recall that  $\Psi_{j-1}(Z)\cap(\bigcup_{i=0}^{j-1}T_i)=\emptyset$ and choose an open neighborhood $U_j$ of $\bigcup_{i=0}^{j-1}T_i$ whose closure is a compact set contained in $E_{j}$ disjoint from $\Psi_{j-1}(Z)\cap E_{j}$. 
Take $\eta>0$ such that $U_j\subset \{\zeta \in E_{j}\colon \dist(\zeta, \c^{N+1}\setminus E_{j})\geq \eta\}$ and such that $(\Psi_{j-1}(Z)\cap \overline E_{j})+\eta \bc$ is disjoint from $U_j$. It then follows that if $\Theta$ is an automorphism of $\c^{N+1}$ such that $|\Theta(\zeta)-\zeta|<\eta$ for all $\zeta\in \overline E_{j}$ then $\Theta(\Psi_{j-1}(Z))\cap U_j=\emptyset$. Thus, {\rm (e$_j$)} is fully satisfied provided that we choose $\epsilon_j<\eta/2$. This closes the induction and completes the proof of the existence of $\Omega$ and $\Psi$ satisfying {\rm (i)} and {\rm (ii)} in Theorem \ref{th:disc}.

The domain $\Omega$ is biholomorphically equivalent to $\bc$ and hence it is pseudoconvex. Since $\overline E_j$ is a compact convex set in $\c^{N+1}$ it is polynomially convex and hence every holomorphic function in a neighborhood of $\overline E_j$ can be, uniformly on $\overline E_j$, approximated by polynomials. Since $\Psi_j$ is an automorphism of $\c^{N+1}$ it follows that every holomorphic function in a neighborhood of $L_j=\Psi_j^{-1}(\overline E_j)$ can be, uniformly on  $L_j$, approximated by entire functions for all $j\in \n$. It follows that $\Omega=\bigcup_{j\in \n} L_j\subset\c^{N+1}$ is a Runge domain. (See \cite{Franc-Book}.)

Notice that $\Omega\cap Z\supset Z\cap\lambda_0\overline\bc\neq\emptyset$ and pick a connected component $D\subset \Omega \cap Z$. Then $D$ is a closed submanifold of $\Omega$. Since $\Omega$ is pseudoconvex it follows that 
$D$ is a pseudoconvex domain in $Z$, and, moreover, by Cartan's extension theorem every holomorphic function $\varphi$ on $D$ extends holomorphically to a holomorphic function $\widetilde\varphi$ on $\Omega$ and, since $\Omega$ is Runge in $\c^{N+1}$, $\widetilde\varphi$ can be, uniformly on compacta in $\Omega$, approximated by entire functions on $\c^{N+1}$. Therefore, $\varphi=\widetilde\varphi|_D$ can be, uniformly on compacta in $D$, approximated by restrictions of entire functions on $\c^{N+1}$ to $Z$. Thus, every holomorphic function on $D$ can be, uniformly on compacta in $D$, approximated by holomorphic functions on $Z$, so $D$ is a Runge domain in $Z$. This completes the proof.
\end{proof}


\section{Proof of the main results}\label{sec:main}

In this final section we make use of Theorems \ref{th:infinity} and \ref{th:disc} in order to prove Theorem \ref{th:main} and Corollary \ref{co:main2}.

\begin{proof}[Proof of Theorem \ref{th:main}]
Let $N\in\n$, let $Z$ be a 
closed complex hypersurface in $\c^{N+1}$ such that $Z\cap\bc_{N+1}\neq\emptyset$, let $K\subset Z\cap \bc_{N+1}$ be a connected compact subset, and let $\epsilon>0$. 

Choose $0<\lambda_0<1$ such that $K\subset \lambda_0 \bc_{N+1}$, let $\Fgot$ be a tidy collection of tangent balls in $\bc_{N+1}$ given by Theorem \ref{th:infinity}, satisfying $|\Fgot|\cap \lambda_0 \overline \bc_{N+1}=\emptyset$.

Apply Theorem \ref{th:disc} to $\Fgot$, $Z$,  $\lambda_0$, and $\epsilon$, and consider the arising Runge pseudoconvex domain $\Omega$ of $\c^{N+1}$, which contains $\lambda_0\overline\bc_{N+1}$, and biholomorphism $\Psi\colon \Omega\to \bc_{N+1}$. Let $D\subset \Omega \cap Z$ be the connected component containing $K$, which ensures {\rm (i)}, and consider the proper holomorphic embedding $\psi:=\Psi|_D\colon D\to \bc_{N+1}$. Item  {\rm (ii)} in Theorem \ref{th:main} follows straightforwardly. To get {\rm (iii)}, observe that the completeness of  $\psi$ is a direct consequence of the choice of $\Fgot$ (see the last sentence in Theorem \ref{th:infinity}), and take into account Theorem \ref{th:disc}-{\rm (i)} and that $K\subset  \lambda_0\overline \bc_{N+1}$. 
\end{proof}

Before proving Corollary \ref{co:main2}, recall that  every open Riemann surface $\Rcal$ is Stein (see Behnke-Stein \cite{BS}), whereas a domain $D\subset\Rcal$ is a Runge domain in $\Rcal$ if and only if $\Rcal\setminus D$ contains no relatively compact connected components; in particular a domain $D\subset \c$ is Runge if and only if it is  simply connected. 

\begin{proof}[Proof of Corollary \ref{co:main2}]
The first part of the corollary trivially follows from Theorem \ref{th:main}.

For the second part, recall first that there are properly embedded complex curves in $\c^2$ with arbitrary topology (see Alarc\'on and L\'opez \cite{AL-Nar}; the case of finite topology, which is required in our proof, is due to \v Cerne and Forstneri\v c \cite{CF}). Let $Z$ be  a connected finite topology properly embedded complex curve in $\c^2$ and assume, up to applying a homothetic transformation to $Z$ if necessary, that all the topology of $Z$ is contained in $\frac12\bc_2$. This means that $Z$ intersects the boundary $\frac12 \s^3$ of $\frac12 \bc_2$ transversely and so that $Z\setminus \frac12 \overline \bc_2$ consists of finitely many  open annuli $A_1,\ldots, A_m$ (here $m$ denotes the number of topological ends of $Z$) with pairwise disjoint closures, properly embedded in $\c^2\setminus \frac12 \overline \bc_2$, such that the boundary $b A_i$  of $A_i$ in $Z$  is a smooth Jordan curve in $\frac12 \s^3$ for all $i=1,\ldots,m$. It follows that $Z\cap \frac12 \bc_2$ is homeomorphic to $Z$ and $K:=Z\cap \frac12 \overline \bc_2$ is a Runge compact connected subset of $Z$. 
Thus, the first part of the corollary provides a Runge domain $D\subset Z$, containing $K$, and a complete proper holomorphic embedding  $\psi\colon D\to\bc_2$. Since $Z\setminus D$ has no relatively compact connected components in $Z$ and $D$ contains $K= Z\cap\frac12\overline\bc_2$, it follows that $D\setminus \frac12 \overline \bc_2$ consists of  finitely many  open annuli $A_1',\ldots, A_m'$ such that $A_i'\subset A_i$ and $bA_i= (bA_i')\cap \frac12 \s^3$ for all $i=1,\ldots,m$. This guarantees that $D$ is homeomorphic to $Z\cap \frac12 \bc_2$, hence to $Z$, which concludes the proof.
\end{proof}

We finish the paper with the following
\begin{question}
We have proved that the unit disc $\d\subset\c$ properly embeds into the unit ball $\bc_2\subset\c^2$ as a complete complex curve, so it is natural to ask whether, given $N\geq 2$, there exists a complete proper holomorphic embedding $\bc_N\hookrightarrow \bc_{N+1}$ (cf.\ \cite{Glob1}).
A less difficult question would be whether there is a complete closed complex hypersurface in $\bc_{N+1}$ which is homeomorphic to $\bc_N$.
\end{question}


\subsection*{Acknowledgements}
The authors wish to thank Jean-Pierre Demailly for suggesting the use of tangent balls instead of faces of convex polytopes. They are also grateful to the referee for very useful comments that led to improvement of the original manuscript.

A. Alarc\'on is supported by the Ram\'on y Cajal program of the Spanish Ministry of Economy and Competitiveness.

A.\ Alarc\'{o}n and F.\ J.\ L\'opez are partially supported by the MINECO/FEDER grant no. MTM2014-52368-P, Spain.

J.\ Globevnik is partially  supported  by the research program P1-0291 and the grant J1-5432 from ARRS, Republic of Slovenia.


%
%
%

\vskip 0.4cm

\noindent Antonio Alarc\'{o}n

\noindent Departamento de Geometr\'{\i}a y Topolog\'{\i}a e Instituto de Matem\'aticas (IEMath-GR), Universidad de Granada, Campus de Fuentenueva s/n, E--18071 Granada, Spain.

\noindent  e-mail: {\tt alarcon@ugr.es}

\vspace*{0.3cm}

\noindent Josip Globevnik

\noindent Department of Mathematics, University of Ljubljana, and Institute
of Mathematics, Physics and Mechanics, Jadranska 19, SI--1000 Ljubljana, Slovenia.

\noindent e-mail: {\tt josip.globevnik@fmf.uni-lj.si}

\vspace*{0.3cm}

\noindent Francisco J.\ L\'opez

\noindent Departamento de Geometr\'{\i}a y Topolog\'{\i}a e Instituto de Matem\'aticas (IEMath-GR), Universidad de Granada, Campus de Fuentenueva s/n, E--18071 Granada, Spain.

\noindent  e-mail: {\tt fjlopez@ugr.es}


\begin{thebibliography}{12}

\bibitem{AF-1}
Alarc\'on, A.; Forstneri\v c, F.: 
Every bordered Riemann surface is a complete proper curve in a ball.
Math.\ Ann. {\bf 357} (2013), 1049--1070

\bibitem{AF-2}
Alarc\'on, A.;  Forstneri\v c, F.: 
Null curves and directed immersions of open Riemann surfaces. 
Invent. Math. \textbf{196}  (2014), 733--771

\bibitem{AF-3}
Alarc\'on, A.;  Forstneri\v c, F.: 
The Calabi-Yau problem, null curves, and Bryant surfaces. 
Math.\ Ann. {\bf  363} (2015), 913--951

\bibitem{AL-CY}
Alarc\'on, A.; L\'opez, F.\ J.: 
Null curves in $\c^3$ and Calabi-Yau conjectures. 
Math.\ Ann. \textbf{355} (2013), 429--455

\bibitem{AL-Nar}
Alarc\'on, A.; L\'opez, F.\ J.: 
Proper holomorphic embeddings of Riemann surfaces with arbitrary topology into $\mathbb{C}^2$.
J.\ Geom.\ Anal. \textbf{23} (2013), 1794--1805

\bibitem{AL-JEMS}
Alarc\'on, A.; L\'opez, F.\ J.: 
Complete bounded embedded complex curves in $\c^2$.
J.\ Eur.\ Math.\ Soc. (JEMS) \textbf{18} (2016), 1675--1705

\bibitem{BS}
Behnke, H.; Stein, K.: 
Entwicklung analytischer Funktionen auf Riemannschen Flächen.
Math.\ Ann.\ \textbf{120} (1949), 430--461

\bibitem{CF}
\v Cerne, M.; Forstneri\v c, F.: 
Embedding some bordered Riemann surfaces in the affine plane. 
Math.\ Res.\ Lett. \textbf{9} (2002), 683--696

\bibitem{D}
Drinovec Drnov\v sek, B.:
Complete proper holomorphic embeddings of strictly pseudoconvex domains into balls.
J.\ Math.\ Anal.\ Appl. {\bf 431} (2015), 705--713

\bibitem{Franc-Book} 
Forstneri\v c, F.:  
Stein Manifolds and Holomorphic Mappings (The Homotopy Principle in Complex Analysis). 
Ergebnisse der Mathematik und ihrer Grenzgebiete, 3.\ Folge, 56. 
Springer-Verlag, Berlin-Heidelberg (2011)

\bibitem{Glob1}
Globevnik, J.:
A complete complex hypersurface in the ball of $\mathbb{C}^N$. 
 Ann.\ of Math. (2) {\bf 182} (2015), 1067--1091

\bibitem{Glob2}
Globevnik, J.:
Holomorphic functions unbounded on curves of finite length. 
Math.\ Ann. \textbf{364} (2016), 1343--1359

\bibitem{Jones}
Jones, P.W.:
 A complete bounded complex submanifold of $\c^3$. 
 Proc.\ Amer.\ Math.\ Soc. {\bf 76} (1979), 305--306

\bibitem{Yang1}
Yang, P.: 
Curvature of complex submanifolds of $\c^n$.
In Several complex variables (Proc.\ Sympos.\ Pure Math., Vol. XXX, Part 2,
Williams Coll., Williamstown, Mass., 1975), Amer. Math. Soc., Providence, R.I., 1977, pp.~135--137.

\bibitem{Yang2}
Yang, P.:  
Curvatures of complex submanifolds of $\c^n$. 
J. Differential Geom. {\bf 12} (1977), 499--511 

\end{thebibliography}
\end{document}